\documentclass[letterpaper,11pt,oneside]{amsart}
\usepackage{amsmath}
\usepackage{amssymb}
\usepackage{latexsym}
\usepackage{amsfonts}
\usepackage[latin1]{inputenc}
\usepackage[english]{babel}
\usepackage{multicol}
\usepackage{oldgerm}
\usepackage{enumerate}
\newtheorem{teo}{Theorem}
\newtheorem{lema}{Lemma}

\def\Q{\mathbb{Q}}

\def\Z{\mathbb{Z}}

\begin{document}

\title{Powers of two as sums of two $k-$Fibonacci numbers}

\author{Jhon J. Bravo}
\address{Departamento de Matem\'aticas\\ Universidad del Cauca\\ Calle 5 No 4--70\\Popay\'an, Colombia.}
\email{jbravo@unicauca.edu.co}

\author{Carlos A. G\'omez}
\address{Departamento de Matem\'aticas\\ Universidad del Valle\\ Calle 13 No 100--00\\Cali, Colombia.}
\email{carlos.a.gomez@correounivalle.edu.co}

\author{Florian Luca}
\address{School of Mathematics, University of the
Witwatersrand, P. O. Box Wits 2050, South Africa and Mathematical Institute, UNAM Juriquilla,  Santiago de Quer\'etaro, 76230  Quer\'etaro de Arteaga, Mexico}
\email{fluca@matmor.unam.mx}


\date{\today}

\begin{abstract}
For an integer $k\geq 2$, let $(F_{n}^{(k)})_{n}$ be the
$k-$Fibonacci sequence which starts with $0,\ldots,0,1$ ($k$
terms) and each term afterwards is the sum of the $k$ preceding
terms.  In this paper, we search for powers of 2 which are sums of
two $k-$Fibonacci numbers. The main tools used in this work are
lower bounds for linear forms in logarithms and a version of the
Baker--Davenport reduction method in diophantine approximation.
This paper continues and extends the previous work of \cite{BL2}
and \cite{BL13}.

\medskip

\noindent\textbf{Keywords and phrases.}\, Generalized Fibonacci numbers, linear forms in logarithms, reduction method.

\noindent\textbf{2010 Mathematics Subject Classification.}\, 11B39, 11J86.

\end{abstract}

\maketitle


\section{Introduction}

\noindent In this paper we consider, for an integer $k\geq 2$, the
\emph{$k-$generalized Fibonacci sequence} or, for simplicity, the
\emph{$k-$Fibonacci sequence} $F^{(k)}:=(F_n^{(k)})_{n\geq 2-k}$
given by the recurrence
\begin{equation}\label{recurrencia}
F_{n}^{(k)}=F_{n-1}^{(k)}+F_{n-2}^{(k)}+\cdots+F_{n-k}^{(k)} \quad \text{for all} \quad n\ge 2,
\end{equation}
with the initial conditions $F^{(k)}_{-(k-2)}=F^{(k)}_{-(k-3)}=\cdots=F^{(k)}_0=0$ and $F^{(k)}_1=1$.

We shall refer to $F_n^{(k)}$ as the \emph{$n$th $k-$Fibonacci number}.
We note that this generalization is in fact a family of sequences
where each new choice of $k$ produces a distinct sequence. For
example, the usual Fibonacci sequence $(F_n)_{n\geq 0}$ is
obtained for $k=2$ and for subsequent values of $k$, these sequences are called Tribonacci, Tetranacci, Pentanacci,
Hexanacci, Heptanacci, Octanacci, and so on.

The first direct observation is that the first $k+1$ non--zero terms in $F^{(k)}$ are powers of two, namely
\begin{equation} \label{primeros-k}
F_{1}^{(k)}=1,~~F_{2}^{(k)}=1, ~~F_{3}^{(k)}=2,~~
F_{4}^{(k)}=4,~\ldots,~F_{k+1}^{(k)}=2^{k-1},
\end{equation}
while the next term in the above sequence is $F_{k+2}^{(k)}=2^{k}-1$. Indeed,
observe that recursion \eqref{recurrencia} implies the three--term recursion
\begin{equation}
\label{three-recursion}
F_n^{(k)}=2F_{n-1}^{(k)}-F_{n-k-1}^{(k)} \quad \text{for all} \quad n\geq 3,
\end{equation}
which also shows that the $k-$Fibonacci sequence grows at a rate less than $2^{n-2}$. We have, in fact, that $F_n^{(k)}<2^{n-2}$ for all $n\geq k+2$
(see \cite[Lemma 2]{BL2}). In addition, Mathematical induction and \eqref{three-recursion} can be used to prove that
\begin{equation} \label{segundos-k}
F_n^{(k)}=2^{n-2}-(n-k)\cdot 2^{n-k-3} \quad \text{holds~ for~ all
}\quad k+2 \leq n \leq 2k+2.
\end{equation}
The above sequences are among the several generalizations of the
Fibonacci numbers, however other generalizations are also known
(see, for example, \cite{Brent,Kilic,Muskat}). The $k-$Fibonacci sequence has been amply studied,
generating an extensive literature.

Recent works on problems involving $k-$Fibonacci numbers are for
instance the papers of F. Luca \cite{FL} and D. Marques \cite{DM},
who proved that 55 and 44 are the largest repdigits (numbers with only one distinct digit) in the
sequences $F^{(2)}$ and $F^{(3)}$, respectively. Moreover, D.
Marques conjectured that there are no repdigits, with at least two
digits, belonging to $F^{(k)}$, for $k > 3$. This conjecture was
confirmed shortly afterwards by Bravo and Luca \cite{BL1}.

Other class of problems has been to represent certain numbers as sum of
$k-$Fibonacci numbers. Regarding this matter, all factorials which are sums of at most three
Fibonacci numbers were found by Luca and Siksek \cite{LS10}; Bravo and Luca \cite{BL10} recently found all
repdigits which are sums of at most two $k-$Fibonacci numbers. Further, the problem of determining all
Fibonacci numbers which are sums of two repdigits is investigated in \cite{SL11}.

In the present paper we extend the works \cite{BL2,BL13} which investigated the powers of 2 that appear in the $k-$Fibonacci sequence
and the powers of 2 which are sums of two Fibonacci numbers, respectively. To be more precise, we study the Diophantine equation
\begin{equation}\label{eqn1}
F_n^{(k)}+F_m^{(k)}=2^a
\end{equation}
in integers $n,m,k$ and $a$ with $k\geq 2$ and $n\geq m$.

Before presenting our main theorem, we observe that in equation
\eqref{eqn1} one can assume $m\geq1$ and $k\geq 3$ since  the other cases were already treated in \cite{BL2,BL13}. Our result is the
following.

\begin{teo} \label{teo1}
Let $(n,m,k,a)$ be a solution of the Diophantine equation \eqref{eqn1} in positive integers $n,m,k$ and $a$ with $k\geq 3$ and $n \geq m$.

\begin{enumerate}
\item[$(a)$] The only solutions of the Diophantine equation \eqref{eqn1} with $n=m$ are given by $(n,m,a)=(1,1,1)$ and $(n,m,a)=(t,t,t-1)$ for all $2\leq t \leq k+1$.

\item[$(b)$] The only solution of the Diophantine equation \eqref{eqn1} with $n>m$ and $a\neq n-2$ is given by $(n,m,a)=(2,1,1)$.

\item[$(c)$] The only solutions of the Diophantine equation \eqref{eqn1} with $n>m$ and $a=n-2$ are given by
   \begin{equation} \label{sol(a=n-2)}
   (n,m,a)=(k+2^{\ell},2^{\ell}+\ell-1,k+2^{\ell}-2),
   \end{equation}
where $\ell$ is a positive integer such that $2^{\ell}+\ell-2\leq k$. So in particular we have $m\le k+1$ and $n\le 2k+1$.
\end{enumerate}
\end{teo}

Let us give a brief overview of our strategy for proving Theorem
\ref{teo1}. The proof of the assertion $(a)$ follows from the work
of \cite{BL2}. To prove assertions $(b)$ and $(c)$, we first rewrite equation \eqref{eqn1} in suitable ways in
order to obtain two different linear forms in logarithms of
algebraic numbers which are both nonzero and small. Next, we use
twice a lower bound on such nonzero linear forms in logarithms of
algebraic numbers due to Matveev \cite{Matveev} to bound $n$ polynomially in
terms of $k$. When $k$ is small, we use some properties of
continued fractions to reduce the upper bounds to cases that can
be treated computationally. When $k$ is large and $a\neq n-2$, we use some
estimates from \cite{BL2,BL1} based on the fact that the dominant
root of $F^{(k)}$ is exponentially close to 2.  However, when $k$ is large and $a=n-2$, the estimates given in \cite{BL2,BL1} are not enough and therefore we need to get more accurate estimates to finish the job.


\section{Preliminary results}

\noindent Before proceeding further, we shall recall some facts
and properties of the $k-$Fibonacci sequence which will be used
later. First, it is known that the characteristic polynomial of
$F^{(k)}$, namely
\[
\Psi_k(x)=x^k-x^{k-1}-\cdots-x-1,
\]
is irreducible over $\mathbb{Q}[x]$ and has just one zero outside
the unit circle. Throughout this paper, $\alpha:=\alpha(k)$
denotes that single zero, which is a Pisot number of degree $k$
since the other zeros of the characteristic polynomial $\Psi_k(x)$
are strictly inside the unit circle (see, for example, \cite{Mi},
\cite{Mil} and \cite{DAW}). Moreover, it is known from Lemma 2.3
in \cite{HY} that $\alpha(k)$ is located between $2(1-2^{-k})$ and
2, a fact rediscovered by Wolfram \cite{DAW}. To simplify
notation, we will omit the dependence on $k$ of $\alpha$.

We now consider for an integer $k\geq 2$, the function
\begin{equation}
\label{def-f_k}
f_k(x)=\frac{x-1}{2+(k+1)(x-2)} \quad \text{for} \quad x>2(1-2^{-k}).
\end{equation}
With this notation, Dresden and Du \cite{D-Du} gave the following ``Binet--like" formula for $F^{(k)}$:
\begin{equation} \label{eq:binetF}
F_n^{(k)}=\sum_{i=1}^{k}f_k({\alpha^{(i)}}){\alpha^{(i)}}^{n-1},
\end{equation}
where $\alpha:=\alpha^{(1)},\ldots,\alpha^{(k)}$ are the zeros of
$\Psi_k(x)$. It was also proved in \cite{D-Du} that the
contribution of the zeros which are inside the unit circle to the
formula \eqref{eq:binetF} is very small, namely that the
approximation
\begin{equation}
\label{error}
\left|F_n^{(k)}-f_k(\alpha)\alpha^{n-1}\right|<\frac{1}{2}  \quad \text{holds~for all}\quad n\geq 2-k.
\end{equation}
When $k=2$, one can easily prove by induction that
\begin{equation} \label{desfibluc}
\alpha^{n-2}\leq F_n\leq \alpha^{n-1} \quad \text{for all} \quad n\geq 1.
\end{equation}
It was proved in \cite{BL1} that
\begin{equation}\label{deskfib}
\alpha^{n-2} \leq F_n^{(k)}\leq \alpha^{n-1} \quad \text{holds for all}\quad n\geq 1\quad \text{and}\quad k\ge 2,
\end{equation}
which shows that \eqref{desfibluc} holds for the $k-$Fibonacci
sequence as well. The observations from expressions \eqref{eq:binetF}
to \eqref{deskfib} lead us to call to $\alpha$ the {\it dominant zero} of $F^{(k)}$.

In order to prove Theorem \ref{teo1}, we need to use several times
a Baker type lower bound for a nonzero linear form in logarithms
of algebraic numbers and such a bound, which plays an important
role in this paper, was given by Matveev \cite{Matveev}. We begin
by recalling some basic notions from algebraic number theory.

Let $\eta$ be an algebraic number of degree $d$ with minimal
primitive polynomial over the integers
\[
a_0x^d+a_1x^{d-1}+\cdots+a_d=a_0\prod_{i=1}^{d}(x-\eta^{(i)}),
\]
where the leading coefficient $a_0$ is positive and the
$\eta^{(i)}$'s are the conjugates of $\eta$. Then
\[
h(\eta)=\frac{1}{d}\left(\log a_0+\sum_{i=1}^{d}\log\left(\max\{|\eta^{(i)}|,1\}\right)\right),
\]
is called the \emph{logarithmic height} of $\eta$.

In particular, if $\eta=p/q$ is a rational number with $\gcd(p,q)=1$ and $q>0$,
then $h(\eta)=\log \max \{|p|,q\}$. The following properties of the function logarithmic height
$h(\cdot)$, which will be used in the next sections without
special reference, are also known:
\begin{eqnarray*}
h(\eta\pm\gamma) & \leq & h(\eta) + h(\gamma)+\log 2,\\
h(\eta\gamma^{\pm 1}) & \leq & h(\eta)+h(\gamma),\\
h(\eta^{s}) & = & |s|h(\eta)\qquad (s\in \Z).
\end{eqnarray*}
With the previous notation, Matveev (see \cite{Matveev} or Theorem 9.4 in \cite{Bug}) proved the following deep theorem.

\begin{teo}[Matveev's theorem]\label{teoMatveev}
Assume that $\gamma_1, \ldots, \gamma_t$ are positive real algebraic numbers in a real algebraic number field $\mathbb{K}$ of degree $D$, $b_1,\ldots,b_t$ are rational integers, and
\[
\Lambda:=\gamma_1^{b_1}\cdots\gamma_t^{b_t}-1,
\]
is not zero. Then
\begin{equation*} \label{desmatveev}
|\Lambda|>\exp\left(-1.4\times 30^{t+3}\times t^{4.5}\times D^2(1+\log D)(1+\log B)A_1\cdots A_t\right),
\end{equation*}
where
\[
B\geq \max\{|b_1|,\ldots,|b_t|\},
\]
and
\[
A_i\geq \max\{Dh(\gamma_i),|\log \gamma_i|, 0.16\}, \quad \text{for all} \quad  i=1,\ldots,t.
\]
\end{teo}

In 1998, Dujella and Peth\H o in \cite[Lemma 5$(a)$]{DP} gave a
version of the reduction method based on the Baker--Davenport
lemma \cite{Baker-Davenport}.  We present the following lemma, which is an immediate variation of the result due to
Dujella and Peth\H o from \cite{DP}, and will be one of the key tools used in this paper to reduce the upper bounds on the
variables of the Diophantine equation \eqref{eqn1}.

\begin{lema} \label{reduce}
Let $M$ be a positive integer, let $p/q$ be a convergent of the continued fraction of the irrational $\gamma$ such that $q>6M$,
and let $A,B,\mu$ be some real numbers with $A>0$ and $B>1$. Let further $\epsilon=||\mu q||-M||\gamma q||$, where $||\cdot||$ denotes the distance from the nearest integer. If $\epsilon >0$, then there is no solution to the inequality
\begin{equation} \label{expDP}
0<|u\gamma-v+\mu|<AB^{-w},
\end{equation}
in positive integers $u,v$ and $w$ with
\[
u\leq M \quad\text{and}\quad w\geq \frac{\log(Aq/\epsilon)}{\log B}.
\]
\end{lema}

\begin{proof}
The proof proceeds similarly to that of Lemma 5 in \cite{DP}. Indeed, assume that $0< u \leq M$. Multiplying \eqref{expDP} by $q$, and keeping in mind that $||q\gamma||=|p-q\gamma|$ because $p/q$ is a convergent of $\gamma$, we get
\begin{align*}
qAB^{-w}&>|q\mu-(qv-up)-u(p-q\gamma)|\\
&\geq |q\mu-(qv-up)|-u|p-q\gamma|\\
&\geq| |q\mu||-u||q\gamma||\\
&\geq ||q\mu||-M||q\gamma||=\epsilon,
\end{align*}
giving
\[
w<\frac{\log(Aq/\epsilon)}{\log B}.
\]
\end{proof}

To conclude this section, we present a useful lemma that will be used later.

\begin{lema}
\label{estimaciondeltaeta} For $k\geq 2$, let $\alpha$ be the dominant root of $F^{(k)}$, and consider the function $f_k(x)$ defined in
\eqref{def-f_k}. Then:

\begin{enumerate}

\item[$(i)$] Inequalities
\[
1/2 < f_k(\alpha) < 3/4 ~~~~~ \text{and} ~~~~~ |f_k(\alpha^{(i)})|< 1, ~~ 2\le i \le k
\]
hold. So, the number $f_k(\alpha)$ is not an algebraic integer.

\item[$(ii)$] The logarithmic height function satisfies $h(f_k(\alpha))<3\log k$.

\end{enumerate}
\end{lema}

\begin{proof}
A straightforward verification shows that $\partial_x f_k (x)<0$.
Indeed,
$$
\partial_x f_k(x)=\frac{1-k}{\left((2+(k+2)(x-2))\right)^2}<0,~~~\text{for~~all}~~~k\ge 2.
$$
From this, we conclude that
\[
{1}/{2}= f_k(2) < f_k(\alpha) <
f_k\left(2(1-2^{-k})\right)=\frac{2^{k-1}-1}{2^k-k-1}\le 3/4, ~~\text{for all}~~ k\ge 3.
\]
While, $f_2((1+{\sqrt{5}})/2) =\sqrt{5}(1+\sqrt{5})/10=0.72360\ldots \in(1/2,3/4)$. On the other hand, as $|\alpha^{(i)}|<1$, then
$|\alpha^{(i)}-1|<2$ and $|2+(k+1)(\alpha^{(i)}-2)|>k-1$, so
$|f_k(\alpha^{(i)})|<1$ for all $k \ge 3$. Further,
$f_2((1-\sqrt{5})/{2}) = 0.2763\ldots$. This proves the first part
of $(i)$. Assume now that $f_k(\alpha)$ is an algebraic integer. Then
its norm (from $\mathbb{K}=\mathbb{Q}(\alpha)$ to $\mathbb{Q}$) is an integer.
Applying the norm of $\mathbb{K}$ over $\mathbb{Q}$ and taking
absolute values, we obtain that
\[
1\le |\mathbf{N}_{\mathbb{K}/\mathbb{Q}}(f_k(\alpha))| =
f_k(\alpha)\prod_{i=2}^k|f_k(\alpha^{(i)})|.
\]
However, $f_k(\alpha) < 3/4$ and $|f_k(\alpha^{(i)})|<1$ for $i=2,\ldots,k$ and all $k\geq 2$, contradicting the above inequality. Hence the result of $(i)$. The proof of $(ii)$ can be consulted in \cite{BL1}.
\end{proof}


\section{An inequality for $n$ in terms of $k$}

\noindent Assume throughout that equation \eqref{eqn1} holds.
First of all, observe that if $n=m$, then the original equation
\eqref{eqn1} becomes $F_n^{(k)}=2^{a-1}$. But the only solutions
of this latter equation with $k\geq 3$ are given by $(n,a)\in\{(1,1),(t,t-1)\}$
for all $2\leq t \leq k+1$ in view of the previous work of
\cite{BL2}. Moreover, note that we can assume $n\geq k+2$, since
otherwise $F_n^{(k)}$ and $F_m^{(k)}$ would be powers of 2, and
therefore, the only additional solution of \eqref{eqn1} in this instance is
given by $(n,m,a)=(2,1,1)$ as can be easily seen. So, from now on,
we assume that $n>m\ge2$, $n\ge k+2$ and $a\ge2$.

Combining \eqref{eqn1} with the fact that $F_t^{(k)}\leq 2^{t-2}$
for all $t\geq 2$, one gets
\[
2^{a} \leq 2^{n-2}+2^{m-2}=2^{n-2}(1+2^{m-n})\leq
2^{n-2}(1+2^{-1})<2^{n-1},
\]
implying $a\leq n-2$. This fact is fundamental in our research to
the point that we distinguish two cases for reasons soon to be seen, namely $a<n-2$ and $a =
n-2$. However, we shall now use linear forms in logarithms to
bound $n$ polynomially on $k$, without any restriction on $a$. 

Indeed, by using \eqref{eqn1} and \eqref{error}, we get that
\begin{equation*} \label{util1}
\left|2^a-f_k(\alpha)\alpha^{n-1}\right|< \frac{1}{2}+F_m^{(k)}\leq  \frac{1}{2}+\alpha^{m-1},
\end{equation*}
where we have also used the right--hand inequality from \eqref{deskfib}.
Dividing both sides of the above inequality by
$f_k(\alpha)\alpha^{n-1}$, we obtain
\begin{equation} \label{deslambda1}
\left|2^{a}\cdot\alpha^{-(n-1)}\cdot
(f_k(\alpha))^{-1}-1\right|<\frac{3}{\alpha^{n-m}},
\end{equation}
because $f_k(\alpha)>1/2$ from Lemma \ref{estimaciondeltaeta}$(i)$.

In a first application of Matveev's result Theorem \ref{teoMatveev}, we take $t:=3$ and
\[
\gamma_1:=2, \quad \gamma_2:=\alpha \quad \text{and} \quad  \gamma_3:=f_k(\alpha).
\]
We also take $b_1:=a$, $b_2:=-(n-1)$ and $b_3:=-1$. We begin by
noticing that the three numbers $\gamma_1,\gamma_2,\gamma_3$ are
positive real numbers and belong to
$\mathbb{K}:=\mathbb{Q}(\alpha)$, so we can take
$D:=[\mathbb{K}:\mathbb{Q}]=k$.  The left--hand size of
\eqref{deslambda1} is not zero. Indeed, if this were zero, we
would then get that $f_k(\alpha)=2^a\cdot \alpha^{-(n-1)}$ and so
$f_k(\alpha)$ would be an algebraic integer, in contradiction to
Lemma \ref{estimaciondeltaeta}$(i)$.

Since $h(\gamma_1)=\log 2$ and $h(\gamma_2)=(\log\alpha)/k<(\log
2)/k=(0.693147\ldots)/k$, it follows that we can take $A_1:=k\log
2$ and  $A_2:=0.7$. Further, in view of Lemma
\ref{estimaciondeltaeta}$(ii)$, we have that $h(\gamma_3) <3\log
k$, so we can take $A_3:=3k\log k$. Finally, by recalling that
$a\leq n-2$, we can take $B:=n-1$.

Then, Matveev's theorem implies that a lower bound on the left--hand side of \eqref{deslambda1} is
\[
\exp\left(-C_1(k)\times (1+\log (n-1))\,(k\log 2)\,(0.7)\,(3k\log k) \right),
\]
where $C_1(k):=1.4\times 30^{6}\times 3^{4.5}\times
k^2\times(1+\log k)<1.5\times 10^{11}\,k^2\,(1+\log k)$. Comparing
this with the right--hand side of \eqref{deslambda1}, taking
logarithms and then performing the respective calculations, we get
that
\begin{equation}\label{aplic1matveev}
(n-m)\log \alpha<8.75\times 10^{11}\,k^4\,\log^2k\,\log(n-1).
\end{equation}
Let us now get a second linear form in logarithms. To this end, we use \eqref{eqn1} and \eqref{error} once again to obtain
\begin{equation} \label{expresion<2}
\left|2^a-f_k(\alpha)\alpha^{n-1}\left(1+\alpha^{m-n}\right)\right|=\left|\left(F_n^{(k)}-f_k(\alpha)\alpha^{n-1}\right)+\left(F_m^{(k)}-f_k(\alpha)\alpha^{m-1}\right)\right|< 1.
\end{equation}
Dividing both sides of the above inequality by the second term of the left--hand side, we get
\begin{equation}    \label{deslambda2}
\left|2^{a}\cdot \alpha^{-(n-1)}\cdot
(f_k(\alpha)(1+\alpha^{m-n}))^{-1}-1\right|<\frac{2}{\alpha^{n-1}}.
\end{equation}
In a second application of Matveev's theorem, we take the parameters $t:=3$ and
\[
\gamma_1:=2, \quad \gamma_2:=\alpha, \quad \gamma_3:=f_k(\alpha)(1+\alpha^{m-n}).
\]
We also take $b_1:=a$, $b_2:=-(n-1)$ and $b_3:=-1$. As before,
$\mathbb{K}:=\Q(\alpha)$ contains $\gamma_1,\gamma_2$ and
$\gamma_3$ and has degree $D:=k$. To see why the left--hand side of
\eqref{deslambda2} is not zero, note that otherwise, we would get
the relation $2^{a}=f_k(\alpha)(\alpha^{n-1}+\alpha^{m-1})$. Now,
conjugating with an automorphism $\sigma$ of the Galois group of
$\Psi_k(x)$ over ${\mathbb Q}$ such that
$\sigma(\alpha)={\alpha^{(i)}} $ for some $i>1$, and then taking
absolute values, we have $2^a=|f_k(\alpha^{(i)})||{\alpha^{(i)}}^{n-1}+{\alpha^{(i)}}^{m-1}|<2$,
since Lemma \ref{estimaciondeltaeta}$(i)$. But the last inequality
above is not possible because $a\ge2$. Hence, indeed the
left--hand side of inequality \eqref{deslambda2} is nonzero.

In this application of Matveev's theorem we take $A_1:=k\log 2$,
$A_2:=0.7$ and $B:=n-1$ as we did before. Let us now estimate
$h(\gamma_3)$. In view  of the properties of $h(\cdot)$ and Lemma
\ref{estimaciondeltaeta}$(ii)$ once again, we have
\begin{align*}
h(\gamma_3) &< 3\log k+|m-n|\left(\frac{\log \alpha}{k}\right)+\log 2 \\
&< 4\log k+(n-m)\left(\frac{\log \alpha}{k}\right),
\end{align*}
for all $k\geq 3$. So, we can take $A_3:=4k\log k+(n-m)\log \alpha$. Now Matveev's theorem implies that a lower bound on the left--hand side of \eqref{deslambda2} is
\[
\exp\left(-C_2(k)\times (1+\log (n-1))\,(k\log 2)\,(0.7)\,(4k\log k+(n-m)\log \alpha) \right),
\]
where $C_2(k):=1.4\times 30^{6}\times 3^{4.5}\times
k^2\times(1+\log k)<1.5\times 10^{11}\,k^2\,(1+\log k)$. So,
inequality \eqref{deslambda2} yields
\[
(n-1)\log \alpha-\log 2 < 2.92\times 10^{11}\,k^3\,\log
k\,\log(n-1)\,(4k\log k+(n-m)\log\alpha).
\]
Using now \eqref{aplic1matveev} in the right--most term of the above inequality and taking into account that $1/\log\alpha<2$, we conclude, after some elementary algebra, that
\begin{equation} \label{deslog1}
n-1<5.12\times 10^{23}\,k^7\,\log^3 k\,\log^2(n-1).
\end{equation}
It easy to check that for $A\geq 100$ the inequality
\[
x<A\log^2 x  \quad\text{implies}\quad x<4A\,\log^2 A.
\]
Thus, taking $A=5.12\times 10^{23}\,k^7\,\log^3k$ and performing
the respective calculations, inequality \eqref{deslog1} yields
$n<6.654\times 10^{27} k^7\,\log^5k$. We record what we have
proved so far as a lemma.

\begin{lema}\label{cota_na}
If $(n,m,k,a)$ is a solution in positive integers of equation \eqref{eqn1} with $n>m\geq 2$ and $k\geq 3$, then both inequalities
\[
a\leq n-2 \quad \text{and} \quad n<6.66\times 10^{27} k^7\,\log^5k
\]
hold.
\end{lema}

\section{Considerations on $k$ for $a<n-2$}

\noindent In this section, we show that for any $k\ge3$, equation
\eqref{eqn1} has no solutions in the range indicated in the title except the one given in Theorem \ref{teo1}$(b)$. A key point for the case $a\neq n-2$ consists of exploiting the
fact that when $k$ is large, the dominant root of $F^{(k)}$ is
exponentially close to 2, so one can write the dominant term of
the Binet formula for $F^{(k)}$ as a power of 2 plus an error
which is well under control. Precisely we will use the
following Lemma from \cite{BL2} (see also \cite{BL1}).

\begin{lema}\label{estimaciondeltaeta2}
For $k\geq 2$, let $\alpha$ be the dominant
root of $F^{(k)}$, and consider the function $f_k(x)$ defined in \eqref{def-f_k}. If $r>1$ is an integer satisfying $r-1<2^{k/2}$,
then
\[
f_k(\alpha)\alpha^{r-1}=2^{r-2}+\frac{\delta}{2}+2^{r-1}\eta+\eta\delta,
\]
where $\delta$ and $\eta$ are real numbers such that
\[
|\delta|<\frac{2^r}{2^{k/2}}\quad  \text{and} \quad
|\eta|<\frac{2k}{2^{k}}.
\]
\end{lema}


\subsection{The case of small $k$} \label{small k-a<n-2}

\noindent We next treat the cases when $k\in[3,340]$. Note that
for these values of the parameter $k$, Lemma \ref{cota_na} gives
us absolute upper bounds for $n$ and $m$. However, these upper
bounds are so large that we wish to reduce them to a range where
the solutions can be identified by using a computer. To do this,
we first let
\begin{equation} \label{eq:z1}
z_1:=a\log 2-(n-1)\log\alpha-\log f_k(\alpha).
\end{equation}
First of all, note that \eqref{deslambda1} can be rewritten as
\begin{equation}  \label{deslambda1(z_1)}
|e^{z_1}-1|<\frac{3}{\alpha^{n-m}}.
\end{equation}
Secondly, by using \eqref{eqn1} and \eqref{error}, we have
\[
f_k(\alpha)\alpha^{n-1}<F_n^{(k)}+\frac{1}{2}<F_n^{(k)}+F_m^{(k)}=2^a.
\]
Consequently, $1<2^a\alpha^{-(n-1)}(f_k(\alpha))^{-1}$ and so $z_1>0$. This, together with \eqref{deslambda1(z_1)}, gives
\[
0<z_1\leq e^{z_1}-1<\frac{3}{\alpha^{n-m}}.
\]
Replacing $z_1$ in the above inequality by its formula \eqref{eq:z1} and dividing both sides of the resulting inequality by $\log\alpha$, we get
\begin{equation} \label{eq:small1}
0<a\left(\frac{\log 2}{\log\alpha}\right)-n+\left(1-\frac{\log
f_k(\alpha)}{\log\alpha}\right)<6\cdot\alpha^{-(n-m)},
\end{equation}
where we have used again the fact that $1/\log\alpha<2$. We put
\[
\hat{\gamma}:=\hat{\gamma}(k)=\frac{\log 2}{\log\alpha},\quad
\hat{\mu}:=\hat{\mu}(k)=1-\frac{\log f_k(\alpha)}{\log\alpha},
\quad A:=6\quad {\text{\rm and}}\quad B:=\alpha.
\]
We also put $M_k:=\left\lfloor 6.66\times 10^{27} k^7 \log^5
k\right\rfloor$, which is an upper bound on $a$ by Lemma
\ref{cota_na}. The fact that $\alpha$ is a unit in ${\mathcal
O}_{\mathbb K}$, the ring of integers of $\mathbb{K}$, ensures
that $\hat{\gamma}$ is an irrational number. Even more,
$\hat{\gamma}$ is transcendental by the Gelfond-Schneider Theorem. Then, the above inequality \eqref{eq:small1} yields
\begin{equation} \label{z1>0}
0<a\hat{\gamma}-n+\hat{\mu}<A\cdot B^{-(n-m)}.
\end{equation}
It then follows from Lemma \ref{reduce}, applied to inequality \eqref{z1>0}, that
\[
n-m <\frac{\log(Aq/\epsilon)}{\log B},
\]
where $q=q(k)>6M_k$ is a denominator of a convergent of the continued fraction of $\hat{\gamma}$ such that $\epsilon=\epsilon(k)=||\hat{\mu}q||-M_k||\hat{\gamma} q||>0$. A computer search with \emph{Mathematica} revealed that if $k\in[3,340]$, then the maximum value of $\log(Aq/\epsilon)/\log B$ is $<$ 680. Hence, we deduce that the possible solutions $(n,m,k,a)$ of the equation \eqref{eqn1} for which $k$ is in the range $[3,340]$ all have $n-m\in[1,680]$.

Let us now work a little bit on \eqref{deslambda2} in order to find an upper bound on $n$. Let
\begin{equation} \label{eq:z2}
z_2:=a \log 2-(n-1)\log\alpha-\log \mu(k,n-m),
\end{equation}
where $\mu(k,n-m):=f_k(\alpha)(1+\alpha^{m-n})$. Therefore, \eqref{deslambda2} can be rewritten as
\begin{equation} \label{lambda2z2}
|e^{z_2}-1|<\frac{2}{\alpha^{n-1}}.
\end{equation}
Note that $z_2\neq 0$; thus, we distinguish the following cases. If $z_2>0$, then $e^{z_2}-1>0$, so from \eqref{lambda2z2} we obtain
\[
0<z_2<\frac{2}{\alpha^{n-1}}.
\]
Suppose now that $z_2<0$. It is a straightforward exercise to
check that $2/\alpha^{n-1}\leq 1/2$ for all $k\geq 3$ and all
$n\geq 5$. Then, from \eqref{lambda2z2}, we have that
$|e^{z_2}-1|<1/2$ and therefore $e^{|z_2|}<2$. Since $z_2<0$, we
have
\[
0<|z_2|\leq e^{|z_2|}-1=e^{|z_2|}|e^{z_2}-1|<\frac{4}{\alpha^{n-1}}.
\]
In any case, we have that the inequality
\[
0<|z_2|<\frac{4}{\alpha^{n-1}}
\]
holds for all $k\geq 3$ and $n\geq 5$. Replacing $z_2$ in the above inequality by its formula \eqref{eq:z2} and arguing as in \eqref{eq:small1}, we conclude that
\begin{equation} \label{smallk-z2}
0<\left|a\left(\frac{\log
2}{\log\alpha}\right)-n+\left(1-\frac{\log\mu(k,n-m)}{\log\alpha}\right)\right|<4\cdot\alpha^{-(n-1)}.
\end{equation}
Here, we also take $M_k:=\left\lfloor 6.66\times 10^{27} k^7
\log^5 k\right\rfloor$ (upper bound on $a$), and, as we explained
before, we apply Lemma \ref{reduce} to inequality
\eqref{smallk-z2}  in order to obtain an upper bound on $n-1$.
Indeed, with the help of \emph{Mathematica} we find that if
$k\in[3,340]$ and $n-m\in[1,680]$, then the maximum value of
$\log(4q/\epsilon)/\log \alpha$ is $<$ 680. Thus, the possible
solutions $(n,m,k,a)$ of the equation \eqref{eqn1} with $k$ in the
range $[3,340]$ all have $n\leq 680$.

Finally, a brute force search with \emph{Mathematica} in the range
\[
3\leq k \leq 340, \quad k+2\leq n \leq 680 \quad \text{and} \quad
2\leq m \leq n-1
\]
gives no solutions for the equation \eqref{eqn1} with $a<n-2$. This completes the analysis in the case $k\in[3,340]$.


\subsection{The case of large $k$}

\noindent Here we assume that $k>340$ and show that \eqref{eqn1}
has no solutions. For such $k$ we have
\[
m<n< 6.66\times 10^{27} k^7 \log^5k<2^{k/2}.
\]
It then follows from Lemma \ref{estimaciondeltaeta2} that
\[
|f_k(\alpha)\alpha^{n-1}-2^{n-2}|<\frac{2^{n-1}}{2^{k/2}}+\frac{2^{n}k}{2^{k}}+\frac{2^{n+1}k}{2^{3k/2}}.
\]
The above inequality obviously holds with $n$ replaced by $m$. This, together with \eqref{expresion<2}, implies
\begin{align*}
\left|2^{n-2}+2^{m-2}-2^a\right|&\le\left|2^{n-2}-f_k(\alpha)\alpha^{n-1}\right|+\left|2^{m-2}-f_k(\alpha)\alpha^{m-1}\right|\notag\\
&+\left|f_k(\alpha)\alpha^{n-1}+f_k(\alpha)\alpha^{m-1}-2^a\right|\notag\\
&<\frac{2^{n-1}+2^{m-1}}{2^{k/2}}+\frac{(2^{n}+2^{m})k}{2^k}+\frac{(2^{n+1}+2^{m+1})k}{2^{3k/2}}+1.
\end{align*}
Dividing both sides of the above inequality by $2^{n-2}$ and taking into account that $n\geq k+2$, we get
\begin{align}
\left|1+2^{m-n}-2^{a-(n-2)}\right|&<\frac{2+2^{m-n+1}}{2^{k/2}}+
\frac{(4+2^{m-n+2})k}{2^k}+\frac{(8+2^{m-n+3})k}{2^{3k/2}}+\frac{1}{2^{n-2}} \notag \\
&<\frac{3}{2^{k/2}}+\frac{6k}{2^k}+\frac{12k}{2^{3k/2}}+\frac{1}{2^{n-2}} \notag\\
&<\frac{6}{2^{k/2}}. \label{util2}
\end{align}
However, the above inequality is not possible when $a<n-2$, since the term of the left hand side is $> 1/2$ because $1+2^{m-n}>1$ and
$2^{a-(n-2)}\le1/2$, in contrast to the right hand side which is very small because $k> 340$.


\section{Considerations on $k$ for $a=n-2$}

\noindent To begin, we note that for any $k\ge3$, the triple
$(n,m,a)=(k+2^{\ell},2^{\ell}+\ell-1,k+2^{\ell}-2)$, where $\ell$
is a positive integer such that $2^{\ell}+\ell-2\leq k$, is a
solution of the Diophantine equation \eqref{eqn1}. Indeed, since
$2^{\ell}+\ell-2\leq k$, we have
\[
2\leq 2^{\ell}+\ell-1\leq k+1 ~~~~ \text{and} ~~~~ k+2\leq k+2^{\ell}\leq 2k+2.
\]
Thus, from \eqref{primeros-k} and \eqref{segundos-k} we get,
respectively
\[
F_{2^{\ell}+\ell-1}^{(k)}=2^{2^{\ell}+\ell-3}\quad \text{and}\quad
F_{k+2^{\ell}}^{(k)}=2^{k+2^{\ell}-2}-2^{2^{\ell}+\ell-3}.
\]
Now it is clear that $F_{k+2^{\ell}}^{(k)}+F_{2^{\ell}+\ell-1}^{(k)}= 2^{k+2^{\ell}-2}$.

We remark that the estimate of Bravo and Luca presented in Lemma \ref{estimaciondeltaeta2} is sufficient for several Diophantine
problems involving $k-$Fibonacci numbers, but for the
case $a=n-2$, that will be discussed below, we require some better
ones. To this end, we recall the following result due to Cooper and Howard \cite{CoHo}.

\begin{lema}\label{teoHoward}
For $k\geq 2$ and $n\geq k+2$,
\[
F_n^{(k)}=2^{n-2}+\sum_{j=1}^{\lfloor \frac{n+k}{k+1} \rfloor-1}C_{n,j} \,2^{n-(k+1)j-2},
\]
where
\[
C_{n,j}=(-1)^j \left[ \binom{n-jk}{j}-\binom{n-jk-2}{j-2}  \right].
\]
\end{lema}
In the above, we have denoted by $\lfloor x \rfloor$ the greatest integer less than or equal to $x$ and used the convention that $\binom{a}{b}=0$ if either $a<b$ or if one of $a$ or $b$ is negative. For example, assuming that $k+2\le n\le 2k+2$ we get $\lfloor (n+k)/(k+1)\rfloor=2$ and $C_{n,1}=-(n-k)$, so Cooper and Howard's formula becomes the identity \eqref{segundos-k}. 


\subsection{The case of small $k$}

Suppose that $k\in[3,690]$. Here, we perform an analysis quite similar to that made in Subsection \ref{small k-a<n-2} to reduce the upper bound on $n$. After doing the respective calculations, we conclude that the possible solutions $(n,m,k,n-2)$ of the equation \eqref{eqn1} with
$k$ in the range $[3,690]$ all have $n\leq 1380$. The procedure is quite similar; hence we omit the details in order to avoid unnecessary repetitions. Finally, a brute force search with \emph{Mathematica} in the range
\[
3\leq k \leq 690, \quad k+2\leq n \leq 1380 \quad \text{and} \quad 2\leq m \leq n-1
\]
confirms the assertion of Theorem \ref{teo1}$(c)$.


\subsection{The case of large $k$} Let us now assume that $k>690$. Note that for these values of $k$ we have
\[
m<n< 6.66\times 10^{27} k^7 \log^5k<2^{k/4}.
\]
We now proceed with the proof of Theorem \ref{teo1}$(c)$ by distinguishing two cases on $n$.

\medskip


\noindent {\bf \underline{Case 1}.} $n \leq 2k+2$. Suppose first that $m\le k+1$. Then, it follows from \eqref{primeros-k} and \eqref{segundos-k}, that
\[
F_n^{(k)}=2^{n-2}-(n-k)\cdot 2^{n-k-3}  \quad \text{and}\quad F_m^{(k)}=2^{m-2}.
\]
Then, from the original equation \eqref{eqn1} we have $2^{m-2}=(n-k)\cdot 2^{n-k-3}$ or, equivalently, $n-k=2^{\ell}$ where $\ell=m-n+k+1$. So, $m=(n-k)+\ell-1=2^{\ell}+\ell-1$. Further, since $m\leq k+1$, we deduce that $2^{\ell}+\ell-2\leq k$. That is, the solution $(n,m,a)$ of the equation \eqref{eqn1} has the shape \eqref{sol(a=n-2)}.

Now suppose $m>k+1$. Note that in this case we have that  $1\leq n-m\leq k$ and $2\leq m-k\leq k+1$ as well as $3\leq n-k\leq k+2$. Here, equation \eqref{eqn1} implies
\[
2^{m-2}=(n-k)\cdot 2^{n-k-3} +(m-k)\cdot 2^{m-k-3}
\]
giving
\[
2^{k+1}=(n-k)\cdot 2^{n-m} +m-k.
\]
Thus, $2^{n-m}\mid m-k$, and consequently, $2^{k+1}\leq (k+2)(k+1)+k+1=(k+1)(k+3)$. This contradicts our assumption that $k>690$.

\medskip


\noindent {\bf \underline{Case 2}.} $n > 2k+2$. To deal with this case, we first remark that a straightforward application of Lemma \ref{teoHoward} allows us to conclude that for all $n\geq k+2$, 
\begin{equation} \label{Fn^k-2term}
F_{n}^{(k)} = 2^{n-2}\left(1-\dfrac{n-k}{2^{k+1}}+s_1\right)~~\text{where}~~|s_1|<\dfrac{4n^2}{2^{2k+2}}.
\end{equation}
Indeed,
\begin{align*}
|s_1| & \le  \sum_{j=2}^{\lfloor \frac{n+k}{k+1} \rfloor-1}\frac{|C_{n,j}|}{2^{(k+1)j}} < \sum_{j\ge2} \frac{2n^j}{2^{(k+1)j}(j-2)!}\\
& < \frac{2n^2}{2^{2k+2}} \sum_{j\ge2} \frac{(n/2^{k+1})^{j-2}}{(j-2)!}  <  \frac{2n^2}{2^{2k+2}}e^{n/2^{k+1}}.
\end{align*}
Further, since $n<2^{k}$ we have that $e^{n/2^{k+1}}<e^{1/2}<2$. Thus
\[
|s_1|<\frac{4n^2}{2^{2k+2}}.
\]
Suppose now that $m\leq k+1$. In this case we use \eqref{eqn1} and \eqref{Fn^k-2term} as well as the fact that $F_m^{(k)}=2^{m-2}$, to obtain
\[
\left|\frac{2^{n-2}(n-k)}{2^{k+1}}-2^{m-2}\right|<2^{n-2}\left|s_1\right|<\dfrac{2^{n}n^2}{2^{2k+2}}.
\]
Dividing the above inequality by $2^{n-2-(k+1)}(n-k)$ and taking into account that $n<2^{k/4}$, we find that
\begin{equation} \label{util-final1}
\left|1-\frac{2^{m-n+k+1}}{n-k}\right|<\frac{1}{2^{k/2}}.
\end{equation}
Note that if the left--hand side of \eqref{util-final1} is not zero, then we deduce $1/n<1/(n-k)<1/2^{k/2}$, which is false since $n<2^{k/4}$ and $k>690$. Hence, $n=k+2^{\ell}$  where $\ell=m-n+k+1$. However, this is impossible since $n-k\in\mathbb{Z}$ and $m-n+k+1\leq -1$.

Finally, suppose that $m\geq k+2$. Note that if one takes $a=n-2$ in \eqref{util2}, then it is clear that $2^{m-n}<6/2^{k/2}<1/2^{k/2-3}$, and so $n-m>k/2-3$. Going back to equality \eqref{eqn1} and substituting $F_n^{(k)}$ and $F_m^{(k)}$, according to the identity \eqref{Fn^k-2term}, we find
\begin{align*}
\left|\frac{2^{n-2}(n-k)}{2^{k+1}}-2^{m-2}\right| & <  2^{n-2}|s_1|+2^{m-2}|s_2|+\frac{2^{m-2}(m-k)}{2^{k+1}}\\
& <  \frac{2^{n+1}n^2}{2^{2k+2}}+\frac{2^{m-2}(m-k)}{2^{k+1}}.
\end{align*}
From the above, and using the facts $n>2k+2$, $n<2^{k/4}$ and $n-m>k/2-3$, we get, after some calculations, that
\begin{equation}\label{util-final2}
\left|1-\frac{2^{m-n+k+1}}{n-k}\right|<\frac{10}{2^{k/2}}.
\end{equation}
Note that the left--hand side of \eqref{util-final2} is zero, since otherwise the same argument used in \eqref{util-final1} leads to
a contradiction. Hence
\begin{equation}\label{equal1-n>2k+2}
\dfrac{2^{n-2}(n-k)}{2^{k+1}}=2^{m-2}.
\end{equation}
In order to exploit the above relation, we shall consider one more term for $F_n^{(k)}$ in the expression \eqref{Fn^k-2term}. Indeed, the same argument that we used to obtain \eqref{Fn^k-2term} allows us to deduce that 
\[
F_{n}^{(k)} =  2^{n-2}\left(1-\dfrac{n-k}{2^{k+1}}+\dfrac{(n-2k-1)(n-2k)-2}{2^{2k+3}}+s_3\right)~~\text{where}~~|s_3|<\frac{4n^3}{2^{3k+3}}.
\]
Combining the above identity for $F_{n}^{(k)}$ and the identity \eqref{Fn^k-2term} applied to $F_{m}^{(k)}$ together with \eqref{eqn1} and the relation \eqref{equal1-n>2k+2}, we conclude that
\[
\left|\frac{2^{n-2}\left((n-2k-1)(n-2k)-2\right)}{2^{2k+3}} - \frac{2^{m-2}(m-k)}{2^{k+1}}\right| < 2^{n-2}|s_3|+2^{m-2}|s_2|<\frac{2^{n+1} n^3}{2^{3k+3}}.
\]
Dividing both sides of the above inequality by $2^{n-2-(2k+3)}$ and using \eqref{equal1-n>2k+2} once again, we get the inequality
\[
|(n-2k-1)(n-2k)-2-2(n-k)(m-k)|<\frac{8n^3}{2^k}<\frac{8}{2^{k/4}}<1,
\] 
and consequently
\begin{equation} \label{equal2-n>2k+2}
(n-2k-1)(n-2k)-2=2(n-k)(m-k).
\end{equation}
On the other hand, by equality \eqref{equal1-n>2k+2} once more we have that $2^{m-n+k+1}=n-k$. From this, we get that $2^{n-m}= 2^{k+1}/(n-k)<2^k$, so $n-m<k$ or equivalently $m-k>n-2k$. Using this fact on equality \eqref{equal2-n>2k+2}, we obtain
\[
(n-2k-1)(n-2k)-2>2(n-k)(n-2k)
\]
implying $(n-2k)(n+1)<-2$, which is impossible because our assumption that $n>2k+2$. This completes the analysis when $n>2k+2$ and $m\geq k+2$ and therefore the proof of Theorem \ref{teo1}.

\medskip

\noindent \textbf{Acknowledgements}. J. J. B. was partially supported by Universidad del Cauca and C. A. G. thanks to the Universidad del Valle for support during his Ph.D. studies. The work of F. L. was supported in part by Projects PAPIIT IN 104512, CONACyT 163787 and a Marcos Moshinsky Fellowship.

\end{document}